%% file: NewPrimalityMulti.tex
\documentclass{amsart}
\usepackage[english]{babel}
\usepackage[T1]{fontenc}
\usepackage[utf8]{inputenc}
\usepackage{amsmath}
\usepackage{amsfonts}
\usepackage{array}

\usepackage{amssymb}
\usepackage{mathrsfs}
\usepackage{amsthm}
\usepackage{xfrac}
\usepackage{lmodern}
\usepackage{fix-cm}
\usepackage{listings}
\usepackage{fancyvrb}
\usepackage{enumerate}   
\usepackage{xcolor}
\usepackage{pgf,tikz}
\usepackage{mathrsfs}
\usetikzlibrary{arrows}
\usepackage{float}
\usepackage{subcaption}

\usepackage{hyperref}{}
\usetikzlibrary{patterns}

\def\NZQ{\Bbb}               
\def\NN{{\NZQ N}}

\def\frk{\frak}               

\def\Phi{{\frk n}}
\def\Phi{{\frk N}}
\def\MC{{\mathcal C}}

\def\MP{{\mathcal P}}

\def\MF{{\mathcal F}}

\def\ML{{\mathcal L}}
\def\MW{{\mathcal W}}

\def\MM{{\mathcal M}}
\def\MH{{\mathcal H}}

\def\opn#1#2{\def#1{\operatorname{#2}}} 
\opn\chara{char} \opn\length{\ell} \opn\pd{pd} \opn\rk{rk}
\opn\projdim{proj\,dim} \opn\injdim{inj\,dim} \opn\rank{rank}
\opn\depth{depth} \opn\grade{grade} \opn\height{height}
\opn\embdim{emb\,dim} \opn\codim{codim}

\opn\Tr{Tr} \opn\bigrank{big\,rank}
\opn\superheight{superheight}\opn\lcm{lcm}
\opn\trdeg{tr\,deg}
\opn\reg{reg} \opn\lreg{lreg} \opn\ini{in} \opn\lpd{lpd}
\opn\size{size}\opn\bigsize{bigsize}
\opn\cosize{cosize}\opn\bigcosize{bigcosize}
\opn\sdepth{sdepth}\opn\sreg{sreg}
\opn\link{link}\opn\fdepth{fdepth}
\opn\div{div} \opn\Div{Div} \opn\cl{cl} \opn\Cl{Cl}
\opn\Spec{Spec} \opn\Supp{Supp} \opn\supp{supp} \opn\Sing{Sing}
\opn\Ass{Ass} \opn\Min{Min}\opn\Mon{Mon} \opn\dstab{dstab} \opn\astab{astab}
\opn\Ann{Ann} \opn\Rad{Rad} \opn\Soc{Soc}
\opn\Im{Im} \opn\Ker{Ker} \opn\Coker{Coker} \opn\Am{Am}
\opn\Hom{Hom} \opn\Tor{Tor} \opn\Ext{Ext} \opn\End{End}
\opn\Aut{Aut} \opn\id{id}

\opn\nat{nat}
\opn\pff{pf}
\opn\Pf{Pf} \opn\GL{GL} \opn\SL{SL} \opn\mod{mod} \opn\ord{ord}
\opn\Gin{Gin} \opn\Hilb{Hilb}\opn\sort{sort}
\opn\aff{aff} \opn\con{conv} \opn\relint{relint} \opn\st{st}
\opn\lk{lk} \opn\cn{cn} \opn\core{core} \opn\vol{vol}
\opn\link{link} \opn\star{star}\opn\lex{lex}
\opn\cdeg{cdeg}
\opn\T{T}
\opn\gr{gr}

\def\pot#1#2{#1[\kern-0.28ex[#2]\kern-0.28ex]}

\opn\dirlim{\underrightarrow{\lim}}
\opn\inivlim{\underleftarrow{\lim}}
\def\Implies{\ifmmode\Longrightarrow \else
        \unskip${}\Longrightarrow{}$\ignorespaces\fi}
\def\implies{\ifmmode\Rightarrow \else
        \unskip${}\Rightarrow{}$\ignorespaces\fi}
\def\iff{\ifmmode\Longleftrightarrow \else
        \unskip${}\Longleftrightarrow{}$\ignorespaces\fi}

\let\:=\colon

\theoremstyle{plain}
\newtheorem{theorem}{Theorem}[section]
\newtheorem{lemma}[theorem]{Lemma}
\newtheorem{proposition}[theorem]{Proposition}
\newtheorem{corollary}[theorem]{Corollary}

\newtheorem{conjecture}[theorem]{Conjecture}

\newtheorem*{theorem-q}{Theorem}
\newtheorem*{corollary-q}{Corollary}
\newtheorem*{question-q}{Question}
\newtheorem*{questions-q}{Questions}
\theoremstyle{definition}
\newtheorem{definition}[theorem]{Definition}
\newtheorem{example}[theorem]{Example}

\theoremstyle{remark}
\newtheorem{remark}[theorem]{Remark}

\newcommand{\W}{\mathcal{W}}

\newcommand{\HH}{\mathcal{H}}
\renewcommand{\phi}{\varphi}

\title{Primality of multiply connected polyominoes}

\author[Carla Mascia]{Carla Mascia}

\address[Carla Mascia]{Department of Mathematics\\
University of Trento\\
via Sommarive, 14\\
38123 Povo (Trento), Italy}
\email{carla.mascia@unitn.it}

\author[Giancarlo Rinaldo]{Giancarlo Rinaldo}

\address[Giancarlo Rinaldo]{Department of Mathematics\\
University of Trento\\
via Sommarive, 14\\
38123 Povo (Trento), Italy}
\email{giancarlo.rinaldo@unitn.it}

\author[Francesco Romeo]{Francesco Romeo}

\address[Francesco Romeo]{Department of Mathematics\\
University of Trento\\
via Sommarive, 14\\
38123 Povo (Trento), Italy}
\email{francesco.romeo-3@unitn.it}

\begin{document}
\input{IntroPre}

\input{ToricZigZag}

\input{Algorithm}
\input{NewProofFG}

\input{Bibliography}
\end{document}

%% file: IntroPre.tex
\begin{abstract}
It is known that the polyomino ideal of simple polyominoes is prime. In this paper, we focus on multiply connected polyominoes, namely polyominoes with holes, and observe that the non-existence of a certain sequence of inner intervals of the polyomino, called zig-zag walk, gives a necessary condition for the primality of the polyomino ideal. Moreover, by computational approach, we prove that for all polyominoes with rank less than or equal to $14$ the above condition is also sufficient. Lastly, we present an infinite new class of prime polyomino ideals.
\end{abstract}

\maketitle

\section{Introduction}
Polyominoes are two-dimensional objects obtained by joining edge by edge squares of same size. They are studied from the point of view of combinatorics, e.g. in tiling problems of the plane, as well as from the point of view of commutative algebra, e.g. associating binomial ideals to polyominoes. The latter have been introduced by Qureshi in \cite{Qu}. In particular, she introduces a binomial ideal generated by the inner $2$-minors of a polyomino, called polyomino ideal. We refer the reader to Section \ref{sec:pre} for the notation.

Two pending and of interest questions regarding polyomino ideals are to classify those that are prime and to prove if they are radical ideals. In this work, we focus on the first question, giving a partial answer in terms of their geometric realization. A polyomino is briefly called prime if its polyomino ideal is prime.
In \cite{HM}, \cite{HQS} and \cite{QSS}, the authors prove that a polyomino is prime if and only if it is balanced, and that the simple polyominoes are prime. A simple polyomino is a polyomino without holes. Whereas, polyominoes having one or more holes are called multiply connected polyominoes, using the terminology adopted in \cite{Go}, an introductory book on polyominoes. 

In general, giving a complete characterization of the primality of multiply connected polyomino ideals seems to be not so easy. A family of prime polyominoes obtained by removing a convex polyomino by a given rectangle has been showed in \cite{HQ} and \cite{Sh1}.
  
In  Section \ref{sec: toric}, we give a necessary condition for the primality of the polyomino ideal with respect to the geometric representation of the polyomino. This condition is related to a sequence of inner intervals contained in the polyomino, called a zig-zag walk (see Definition \ref{Def:zigzag}), whose existence determines the non-primality of the polyomino ideal. 

It is known that a polyomino ideal that is prime is a toric ideal. We present a toric ideal associated to a polyomino, generalizing Shikama's construction in \cite{Sh1}. This toric ideal contains the polyomino ideal (see Proposition \ref{Lemma: I_P subseteq J_P}). Moreover, if the polyomino contains a zig-zag walk, the binomial associated to the zig-zag walk belongs to the toric ideal and the above inclusion is strict.

The condition on zig-zag walks gives us a good filtration of primality. In fact, as an application, by implementing the algorithm described in \cite{MRR}, we compute all the polyominoes with rank less than or equal to 14 that are 123851 (for a complete description of the algorithm see \cite{MRR}). By computational approach, using \texttt{Macaulay2} \cite{M2}, we obtain the following 
  \begin{theorem}
  Let $\MP$ be a polyomino with $\rank(\MP)$ $\leq 14$. The following conditions are equivalent:
  \begin{enumerate}    
  \item the polyomino ideal $I_\MP$ is prime;
  \item $\MP$ contains no zig-zag walk.
  \end{enumerate}
  \end{theorem}
  
In the last section, we observe that removing $5$ squares in a particular position from a given rectangle, we obtain a polyomino with a zig-zag walk (see Figure \ref{fig:grid} (B)). Moreover, we define a new infinite family of polyominoes that we call grid polyominoes, that are obtained by removing rectangular holes by a given rectangle in a way that avoids the existence of zig-zag walks. We prove that grid polyominoes are primes.

Therefore, the natural conjecture arises:
\begin{conjecture}
  Let $\MP$ be a polyomino. The following conditions are equivalent:
  \begin{enumerate}      
    \item the polyomino ideal $I_\MP$ is prime;
  \item $\MP$ contains no zig-zag walks.
\end{enumerate}
\end{conjecture}

\section{Preliminaries}\label{sec:pre}

In this section we recall definitions and notation first introduced by Qureshi in \cite{Qu}. Let $a = (i, j), b = (k, \ell) \in \NN^2$, with $i	\leq k$ and $j\leq\ell$, the set $[a, b]=\{(r,s) \in \NN^2 : i\leq r \leq k \text{ and } j \leq s \leq \ell\}$ is called an \textit{interval} of $\NN^2$. If $i<k$ and $j < \ell$, $[a,b]$ is called a \textit{proper interval}, and the elements $a,b,c,d$ are called corners of $[a,b]$, where $c=(i,\ell)$ and $d=(k,j)$. In particular, $a,b$ are called \textit{diagonal corners} and $c,d$ \textit{anti-diagonal corners} of $[a,b]$. The corner $a$ (resp. $c$) is also called the left lower (resp. upper) corner of $[a,b]$, and $d$ (resp. $b$) is the right lower (resp. upper) corner of $[a,b]$. 
A proper interval of the form $C = [a, a + (1, 1)]$ is called a \textit{cell}. Its vertices $V(C)$ are $a, a+(1,0), a+(0,1), a+(1,1)$ and its edges $E(C)$ are 
\[
 \{a,a+(1,0)\}, \{a,a+(0,1)\},\{a+(1,0),a+(1,1)\},\{a+(0,1),a+(1,1)\}.
\]
Let $\MP$ be a finite collection of cells of $\NN^2$, and let $C$ and $D$ be two cells of $\MP$. Then $C$ and $D$ are said to be \textit{connected}, if there is a sequence of cells $C = C_1,\ldots, C_m = D$ of $\MP$ such that $C_i\cap C_{i+1}$ is an edge of $C_i$
for $i = 1,\ldots, m - 1$. In addition, if $C_i \neq C_j$ for all $i \neq j$, then $C_1,\dots, C_m$ is called a \textit{path} (connecting $C$ and $D$). A collection of cells $\MP$ is called a \textit{polyomino} if any two cells of $\MP$ are connected. We denote by $V(\MP)=\cup _{C\in \MP} V(C)$ the vertex set of $\MP$. The number of cells of $\MP$ is called the \textit{rank} of $\MP$, and we denote it by $\rank(\MP)$.

 A proper interval $[a,b]$ is called an \textit{inner interval} of $\MP$ if all cells of $[a,b]$ belong to $\MP$.
We say that a polyomino $\MP$ is \textit{simple} if for any two cells $C$ and $D$ of $\NN^2$ not belonging to $\MP$, there exists a path $C=C_{1},\dots,C_{m}=D$ such that $C_i \notin \MP$ for any $i=1,\dots,m$.  If the polyomino is not simple then it is multiply connected (see \cite{Go}).

A finite collection $\MH$ of cells not in $\MP$ is called a \emph{hole} of $\MP$, if any two cells  in $\MH$ are connected through a path of cells in $\MH$, and $\MH$ is maximal with respect to the inclusion. Note that a hole $\MH$ of a polyomino $\MP$ is itself a simple polyomino.

Following \cite{HQS}, an interval $[a,b]$ with $a = (i,j)$ and $b = (k, \ell)$ is called a \textit{horizontal edge interval} of $\MP$ if $j =\ell$ and the sets $\{(r, j), (r+1, j)\}$ for
$r = i, \dots, k-1$ are edges of cells of $\MP$. If a horizontal edge interval of $\MP$ is not
strictly contained in any other horizontal edge interval of $\MP$, then we call it \textit{maximal horizontal edge interval}. Similarly, one defines vertical edge intervals and maximal vertical edge intervals of $\MP$.

Let $a=(a_1,a_2)$ and $b=(b_1,b_2)\in V(\MP)$, we define on the vertices of $\MP$ the following total order: $a < b$ if $a_1<b_1$ or $a_1=b_1$ and $a_2<b_2$.

Let $\MP$ be a polyomino, and let $\mathbb{K}$ be a field. We denote by $S$ the polynomial over $\mathbb{K}$ with variables $x_v$, where $v \in V (\MP)$. The binomial $x_a x_b - x_c x_d\in S$ is called an \textit{inner 2-minor} of $\MP$ if $[a,b]$ is an inner interval of $\MP$, where $c,d$ are the anti-diagonal corners of $[a,b]$. We denote by $\MM$ the set of all inner 2-minors of $\MP$. The ideal $I_\MP\subset S$ generated by $\MM$ is called the \textit{polyomino ideal} of $\MP$. We also set $\mathbb{K} [\MP] = S/I_\MP$ .

%% file: ToricZigZag.tex
\section{The toric ring of generic Polyominoes and zig-zag walks}\label{sec: toric}
Let $\MP$ be a polyomino. Let $S=\mathbb{K}[x_v | v \in V(\MP)]$ and $I_{\MP} \subset S$ the polyomino ideal associated to $\MP$. 
Let $\MH$ be a hole of $\MP$. We call \textit{lower left corner} $e$ of $\HH$ the minimum, with respect to $<$, of the vertices of $\HH$. 

Let $\HH_1, \dots, \HH_r$ be holes of $\MP$. For $k=1, \dots, r$, we denote by $e_k =(i_k,j_k)$ the lower left corner of $\HH_k$. For $k \in K =\{1, \dots, r\}$, we define the following subset of $V(\MP)$
\[
\MF_k = \{(i,j) \in V(\MP) \ \mid \ i \leq i_k \text{ and } j \leq j_k\}.
\]
Let $\{V_i\}_{i \in I}$ be the set of all the maximal vertical edge intervals of $\MP$, and $\{H_j\}_{j \in J}$ be the set of all the maximal horizontal edge intervals of $\MP$. Let $\{v_i\}_{i \in I}, \{h_j\}_{j \in J}$, and $\{w_k\}_{w \in K}$ be three sets of variables associated to $\{V_i\}_{i \in I}, \{H_j\}_{j \in J}$, and $\{\MF_k\}_{k\in K}$, respectively.
We consider the map
\begin{align*}
\alpha: V(\MP) &\longrightarrow \mathbb{K}[\{h_i, v_j,w_k\}\mid i \in I, j\in J, k \in K] \\
a & \longmapsto  \prod_{a \in H_i \cap V_j} h_i v_j \prod_{a \in \MF_k} w_k
\end{align*}
The \textit{toric ring} $T_\MP$ associated to $\MP$  is defined as $T_{\MP} = \mathbb{K}[ \alpha(a) | a \in V(\MP)] \subset \mathbb{K}[\{h_i, v_j,w_k\}\mid i \in I, j\in J, k \in K]$. The homomorphism 
\begin{align*}
\varphi: S &\longrightarrow T_{\MP} \\
x_a &\longmapsto \alpha(a)
\end{align*}
is surjective and the \textit{toric ideal} $J_{\MP}$ is the kernel of $\varphi$. The toric ring $T_{\MP}$ is viewed as a standard graded $\mathbb{K}$-algebra and, therefore, the corresponding toric ideal $J_{\MP}$ is standard graded.

By definition, $J_\MP$ is a prime ideal containing $I_\MP$. Moreover, the next result shows that for any polyomino $\MP$, $(J_{\MP})_2$, the homogeneous part of degree 2 of $J_{\MP}$, is equal to $I_{\MP}$, that means that the minimal generators of $I_{\MP}$ are all and only the minimal generators of degree 2 of $J_{\MP}$.

\begin{lemma}\label{Lemma: I_P subseteq J_P}
Let $\MP$ be a polyomino. Then $I_{\MP} =(J_{\MP})_2$.
\end{lemma}

\begin{proof}
First of all we show that $I_{\MP} \subseteq (J_{\MP})_2$. Let $f \in \MM$, with $f=x_ax_b -x_cx_d$. Since $[a,b]$ is an inner interval of $\MP$, the corners $a$ and $d$ (resp. $b$ and $c$) lie on the same horizontal edge interval $H_i$ (resp. $H_j$). In the same way, it holds that $a$ and $c$ (resp. $b$ and $d$) lie on the same vertical edge interval $V_l$ (resp. $V_m$). Therefore,
\begin{equation}\label{Eq: varphi x_ax_b}
\varphi(x_ax_b) = h_ih_jv_lv_m\prod_{k=1,\dots,r} w_k^{p_k}
\end{equation}
and 
\begin{equation}\label{Eq: varphi x_cx_d}
\varphi(x_cx_d) = h_ih_jv_lv_m \prod_{k=1,\dots,r}w_k^{n_k}
\end{equation}
for some $p_k, n_k \in \{0,1,2\}$. We have to show that for any $k \in \{1,\ldots ,r\}$ $p_k=n_k$.
If $\MP$ has not holes then $\varphi(x_ax_b) = \varphi(x_cx_d)$, and $f \in J_{\MP}$. Suppose that $\MH_1, \dots, \MH_r$ are holes of $\MP$ and consider $\MH_k$ for $k=1, \dots,r$.
Observe that the left lower corner $e_k$ of $\MH_k$ satisfies one of the following
\begin{enumerate}
\item $e_k < a$; 
\item $a \leq e_k \leq d$;
\item $d < e_k$.
\end{enumerate}
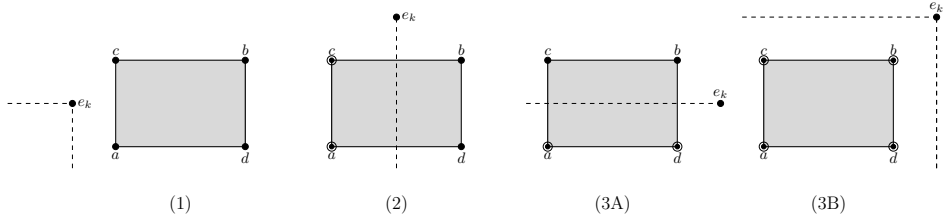
\begin{figure}[h]
\centering
   \resizebox{1 \textwidth}{!}{
  \begin{tikzpicture}
   \draw[thick] (0-5,0)--(3-5,0);
  \draw[thick] (0-5,0)--(0-5,2);
  \draw[thick] (3-5,2)--(3-5,0);
   \draw[thick] (3-5,2)--(0-5,2);
  
  \fill[fill=gray, fill opacity=0.3] (0-5,0) -- (3-5,0)-- (3-5,2) -- (0-5,2);
   \filldraw (0-5,0) circle (2pt) node [anchor=north] {$a$};
  \filldraw(3-5,0) circle (2pt) node [anchor=north] {$d$};
  \filldraw  (0-5,2) circle (2pt) node [anchor=south] {$c$};
  \filldraw (3-5,2) circle (2pt) node [anchor=south] {$b$};
  
   \draw[dashed] (-1-5,-0.5)-- (-1-5,1);
   \draw[dashed] (-1-5-1.5,1)-- (-1-5,1);
   \filldraw (-1-5,1) circle (2pt) node [anchor=west] {$e_k$};   
    \draw (-5+1.5,-1) circle (0pt) node [anchor=north] {\Large (1)};

   \draw[thick] (0,0)--(3,0);
  \draw[thick] (0,0)--(0,2);
  \draw[thick] (3,2)--(3,0);
   \draw[thick] (3,2)--(0,2);
  
  \fill[fill=gray, fill opacity=0.3] (0,0) -- (3,0)-- (3,2) -- (0,2);
\filldraw (0,0) circle (1.5pt) node [anchor=north] {$a$};
  \filldraw(3,0) circle (2pt) node [anchor=north] {$d$};
\filldraw  (0,2) circle (1.5pt) node [anchor=south] {$c$};
  \filldraw (3,2) circle (2pt) node [anchor=south] {$b$};
  \draw (0,0) circle (3pt);
  \draw  (0,2) circle (3pt);
  \draw[dashed] (1.5,-0.5)-- (1.5,3);
   \filldraw (1.5,3) circle (2pt) node [anchor=west] {$e_k$};   
     \draw (1.5,-1) circle (0pt) node [anchor=north] {\Large (2)};

  \draw[thick] (0+5,0)--(3+5,0);
  \draw[thick] (0+5,0)--(0+5,2);
  \draw[thick] (3+5,2)--(3+5,0);
   \draw[thick] (3+5,2)--(0+5,2);

  \fill[fill=gray, fill opacity=0.3] (0+5,0) -- (3+5,0)-- (3+5,2) -- (0+5,2);
\filldraw (0+5,0) circle (1.5pt) node [anchor=north] {$a$};
\filldraw (3+5,0) circle (1.5pt) node [anchor=north] {$d$};
  \filldraw (0+5,2) circle (2pt) node [anchor=south] {$c$};
  \filldraw (3+5,2) circle (2pt) node [anchor=south] {$b$};
\draw (0+5,0) circle (3pt);
\draw (3+5,0) circle (3pt);  
  
  \draw (1.5+5,-1) circle (0pt) node [anchor=north] {\Large (3A)};
 
  \draw[dashed] (-0.5+5,1)-- (4+5,1);
   \filldraw (4+5,1) circle (2pt) node [anchor=south] {$e_k$};
 
    \draw[thick] (0+10,0)--(3+10,0);
  \draw[thick] (0+10,0)--(0+10,2);
  \draw[thick] (3+10,2)--(3+10,0);
   \draw[thick] (3+10,2)--(0+10,2);
  
  \fill[fill=gray, fill opacity=0.3] (0+10,0) -- (3+10,0)-- (3+10,2) -- (0+10,2);
  \filldraw (0+10,0) circle (1.5pt) node [anchor=north] {$a$};
\filldraw (3+10,0) circle (1.5pt) node [anchor=north] {$d$};
\filldraw (0+10,2) circle (1.5pt) node [anchor=south] {$c$};
\filldraw (3+10,2) circle (1.5pt) node [anchor=south] {$b$};
\draw (0+10,0) circle (3pt);
\draw (3+10,0) circle (3pt) ;
\draw (0+10,2) circle (3pt);
\draw (3+10,2) circle (3pt);
  \draw[dashed] (4+10,-0.5)-- (4+10,3);
  \draw[dashed] (4+10,-0.5)-- (4+10,3);
  \draw[dashed] (-0.5+10,3)-- (4+10,3);
   \filldraw (4+10,3) circle (2pt) node [anchor=south] {$e_k$};
    \draw (1.5+5+5,-1) circle (0pt) node [anchor=north] {\Large (3B)};
\end{tikzpicture}  
  }
\caption{Some positions of $e_k$ and induced flagging on $[a,b]$}\label{fig:e_k}

 \end{figure}
Case $(1)$. $w_k$ does not divide $\varphi(f)$ (see Figure \ref{fig:e_k} (1)). Case $(2)$.  $w_k$ divides either both $\varphi(x_a)$ and $\varphi(x_c)$ (see Figure \ref{fig:e_k} (2)) or it does not divide neither $\varphi(x_ax_b)$ nor $\varphi(x_cx_d)$. Case $(3)$. $w_k$ divides either $\varphi(x_a)$ and $\varphi(x_d)$ (see Figure \ref{fig:e_k} (3A)) or all $\varphi(x_a), \varphi(x_b), \varphi(x_c)$ and $\varphi(x_d)$ (see Figure \ref{fig:e_k} (3B)) or $w_k$ does not divide neither $\varphi(x_ax_b)$ nor $\varphi(x_cx_d)$. Therefore $n_k = p_k$, and it holds for any $k=1,\dots, r$.  It follows $\varphi(x_ax_b) = \varphi(x_cx_d)$, and $f \in \ker \varphi = J_{\MP}$. Since all generators of $I_{\MP}$ belong to $J_{\MP}$, the inclusion $I_{\MP} \subseteq (J_{\MP})_2$ is proved. 

We are going to prove the other inclusion, namely $(J_{\MP})_2 \subseteq I_{\MP}$. Let $f \in J_{\MP}$, $f=x_ax_b-x_cx_d$. We start observing that if $a=b$ or $a\in \{c,d\}$  we obtain that $f$ is null. Hence we assume without loss of generality  $a<b$ and $c<d$. Since $\varphi(x_ax_b)=\varphi(x_cx_d)$, by \eqref{Eq: varphi x_ax_b} and \eqref{Eq: varphi x_cx_d} the vertices $a$ and $d$ (resp. $b$ and $c$) lie on the same horizontal edge interval of $\MP$, and $a$ and $c$ (resp. $b$ and $d$) lie on the same vertical edge interval of $\MP$, and all the vertices of these edge intervals belong to $\MP$. Therefore, the vertices $a,b,c$, and $d$ are the corners of the interval $[a,b]$. By contradiction, we assume that $[a,b]$ is not an inner interval of $\MP$, namely exists a set of cells $\MC$ that does not belong to $\MP$ such that $[a,b] \cap \MC \neq \emptyset$. 
We observe that the set  $[a,b] \cap \MC$ is a set of holes of $\MP$ properly contained in $[a,b]$ because $[a,d]$, $[a,c]$, $[b,c]$ and  $[b,d]$ are edge intervals in $\MP$.  Let $\MH_1$ be a hole in $[a,b] \cap \MC$ with lower left corner $e=(i,j)$. Let $\MF_1=\{(m,n) \in V(\MP) \ | \ m \leq i \text{ and } n \leq j\}$, then $a$ is the unique vertex in $\{a,b,c,d\}$ such that $a \in \MF_1$, namely $w_1 |\varphi(x_ax_b)$ but $w_1 \nmid \varphi(x_cx_d)$, and $f \notin J_{\MP}$. The assertion follows.
\end{proof}

%

Describing completely the elements of $J_\MP \setminus I_\MP$ is not an easy task. However, if the polyomino contains a particular collection of inner intervals, then we have some partial information on the elements of $J_\MP \setminus I_\MP$. The latter gives also a sufficient condition for the non-primality of $I_\MP$, hence a necessary condition for the primality.  In the rest of the section, we give such a condition.

\begin{definition}\label{Def:zigzag}
Let $\MP$ be a polyomino. A sequence of distinct inner intervals $\W: I_1, \dots, I_\ell$ of $\MP$ such that $v_i$, $z_i$ are diagonal (resp. anti-diagonal) corners and $u_i$, $v_{i+1}$ the anti-diagonal (resp. diagonal) corners of $I_i$, for $i=1, \dots, \ell$,  is a {\em zig-zag walk} of $\MP$, if 
\begin{enumerate}
 \item[(Z1)] $I_1 \cap I_\ell = \{v_1=v_{\ell+1}\}$ and $I_i \cap I_{i+1}=\{v_{i+1}\}$, for $i =1, \dots, \ell-1$,
 \item[(Z2)] $v_i$ and $v_{i+1}$ are on a same edge interval of $\MP$, for $i=1, \dots,\ell$,
 \item[(Z3)] for any $i,j\in\{1,\ldots,\ell\}$, with $i \neq j$, does not exist an inner interval $J$ of $\MP$ such that $z_i,z_j\in J$.
\end{enumerate}
\end{definition}

\begin{remark}\label{Rem: ell odd or even} 
Let $\W:I_1 \dots, I_\ell$ be a zig-zag walk of $\MP$. Then 
\begin{enumerate}
 \item if $v_i$ is a diagonal vertex of $I_i$, then $v_{i+1}$ is an anti-diagonal vertex of $I_{i+1}$;
 \item $\ell$ is even.
\end{enumerate}
\end{remark}
\begin{proof}
(1) Assume that $v_k$, with $k \in \{1,\ldots, \ell-1\}$ is a diagonal corner of $I_k$. From condition (Z2), $v_{k+1}$ lies on the same edge interval of $v_k$, say $E$, and is an anti-diagonal corner of $I_k$. The line containing $E$ divides $\mathbb{N}^2$ in two semi-planes. From condition (Z1), we have $I_k \cap I_{k+1} =\{v_{k+1}\}$, hence $I_k$ and $I_{k+1}$ do not lie on the same semi-plane. Therefore, $v_{k+1}$ is anti-diagonal corner of $I_{k+1}$, as well. Observe that the latter justifies the name ``zig-zag''. \\
(2) Assume that the starting point $v_1$ is a diagonal corner of $I_1$.  From (1) it follows that the vertex $v_k$ is a diagonal corner of $I_k$ if and only if $k$ is even (resp. anti-diagonal corner if and only if $k$ is odd). Since $v_{\ell+1}=v_1$, $\ell+1$ is odd.
\end{proof}

\begin{remark}\label{Rem: supp f in I_P}
Let $\MP$ be a polyomino and $I_{\MP}\subset S$ the polyomino ideal associated to $\MP$. If $f \in I_{\MP}$, then 
\[
f=\sum f_{I_i}f_i =\sum x_{a_i}x_{b_i}f_i-\sum x_{c_i}x_{d_i}f_i,
\]
where $f_{I_i}=x_{a_i}x_{b_i}-x_{c_i}x_{d_i}\in \MM$, hence for every $m$, monomial of $f$, there are two variables in $m$ that are (anti-)diagonal corners of an inner interval of $\MP$. 
\end{remark}

The following proposition gives a necessary condition on $\MP$ to be a non-prime polyomino ideal $I_{\MP}$. 

\begin{proposition}\label{Prop: if exists zig-zag then not prime}
Let $\MP$ be a polyomino and $I_{\MP}$ the polyomino ideal associated to $\MP$. If there exists a zig-zag walk  $\W: I_1, \dots, I_{\ell}$  in $\MP$ then
\[
x_{v_1},\ldots,x_{v_\ell}\mbox{ and }f_\W = \prod_{k=1,\dots,\ell} x_{z_k} - \prod_{j=1,\dots,\ell} x_{u_j}
\] 
are zerodivisors of $K[\MP]$ with $x_{v_i}f_\W\in I_\MP$ for $i=1,\ldots,\ell$.
\end{proposition}

\begin{proof}
For any vertex $v_j$ in $v_1,\ldots,v_\ell$, after relabelling, we may assume $j=1$.  Let $f_{I_i} \in \MM$ be associated to the inner interval $I_i$.

We define the following polynomial
\[
\tilde{f}= \prod_{k>1} x_{z_k} f_{I_1} + \dots + (-1)^{i+1}
\prod_{ j<i} x_{u_j} \prod_{k>i} x_{z_k} f_{I_i} +\dots + (-1)^{\ell+1} \prod_{j<\ell} x_{u_j} f_{I_\ell}
\]
Let $i=1, \dots, \ell-1$. Suppose that $v_{i}$ is a diagonal corner of $I_i$, hence $v_{i+1}$ is an anti-diagonal corner of $I_{i+1}$. It holds
\begin{eqnarray*}
\prod_{ j<i} x_{u_j} \prod_{k>i} x_{z_k} f_{I_i}&-&\prod_{ j<i+1} x_{u_j} \prod_{k>i+1} x_{z_k} f_{I_{i+1}}\\
= \prod_{j<i} x_{u_j} \prod_{ k>i} x_{z_k} (x_{v_i}x_{z_i}-x_{v_{i+1}}x_{u_i})&-&\prod_{j<i+1} x_{u_j} \prod_{k>i+1} x_{z_k} (x_{v_{i+2}}x_{u_{i+1}}-x_{v_{i+1}}x_{z_{i+1}})\\
= \prod_{ j<i} x_{u_j} \prod_{ k \geq i} x_{z_k} x_{v_i}&-&\prod_{ j \leq i+1} x_{u_j} \prod_{ k>i+1} x_{v_{i+2}}.
\end{eqnarray*}
Due to the alternation of the signs in $\tilde{f}$ and by Remark \ref{Rem: ell odd or even}, it follows that
\[
\tilde{f} = \pm \left( \prod_{k=1,\dots,\ell} x_{z_k}x_{v_1} - \prod_{j=1,\dots,\ell} x_{u_j} x_{v_1} \right) = \pm x_{v_1}f_\W,
\]
and the sign of $\tilde{f}$ depends if $v_1$ is a diagonal corner in $I_1$.

Since $\tilde{f}$ is sum of polynomials in $I_{\MP}$, then $\tilde{f} \in I_{\MP}$. Observe that, by hypothesis, for $i \neq j$, $z_i, z_j$ do not belong to the same inner interval of $\MP$, and the same fact holds for $u_i$ and $u_j$,  with $i\neq j$.  Due to this fact and by Remark \ref{Rem: supp f in I_P}, $f \not \in I_{\MP}$. Therefore, $x_{v_1}$ and $f_\W$ are  zerodivisors  of $K[\MP]$. 
\end{proof}
\begin{corollary}\label{Cor: if zig-zag then non-prime}
Let $\MP$ be a polyomino and $I_{\MP}$ the polyomino ideal associated to $\MP$. If there exists a zig-zag walk  in $\MP$, then $I_\MP$ is not prime.
\end{corollary}

\begin{remark}
The ideal $J_\MP$ contains the binomials associated to zig-zag walks. Indeed, let $\MW$ be a zig-zag walk and let $f_\MW$ be its associated binomial. From the proof of Proposition \ref{Prop: if exists zig-zag then not prime}, it arises that 
\[
x_{v_1} f_\MW \in I_\MP \subseteq J_\MP
\]
and,  due to primality of $J_\MP$, it follows $f_\MW \in J_\MP$.
\end{remark}

We give an example to better understand the structure of $J_\MP$.
\begin{example}
We consider the polyomino in Figure \ref{Fig: 2Ears}.
 \begin{figure}[h!]
  \centering
    \resizebox{0.4\textwidth}{!}{
  \begin{tikzpicture}

\draw[thick] (1,0) --  (6,0);
\draw[thick] (0,1) --  (7,1);
\draw[thick] (0,2) --  (7,2);
\draw[thick] (1,3) --  (6,3);

\draw[thick] (0,1) --  (0,2);
\draw[thick] (1,0) --  (1,3);
\draw[thick] (2,0) --  (2,3);
\draw[thick] (3,0) --  (3,3);
\draw[thick] (4,0) --  (4,3);
\draw[thick] (5,0) --  (5,3);
\draw[thick] (6,0) --  (6,3);
\draw[thick] (7,1) --  (7,2);

\fill[fill=gray, fill opacity=0.3] (1,0) -- (6,0)-- (6,1) -- (1,1);
\fill[fill=gray, fill opacity=0.3] (1,2) -- (6,2)-- (6,3) -- (1,3);
\fill[fill=gray, fill opacity=0.3] (0,1) -- (2,1)-- (2,2) -- (0,2);
\fill[fill=gray, fill opacity=0.3] (3,1) -- (4,1)-- (4,2) -- (3,2);
\fill[fill=gray, fill opacity=0.3] (5,1) -- (7,1)-- (7,2) -- (5,2);

   \end{tikzpicture}}
      \caption{}\label{Fig: 2Ears}

   \end{figure}   
By using \texttt{Macaulay2}, we computed the ideal $J_\MP$ associated to $\MP$. $J_\MP$ has $50$ generators, $46$ having degree $2$, corresponding to the inner $2$-minors of $\MP$, and $4$ having degree $4$ that do not belong to $I_\MP$. The latter are:
\[
f_1=x_{(1,3)}x_{(3,1)}x_{(7,4)}x_{(8,2)} -x_{(1,2)}x_{(3,4)}x_{(7,1)}x_{(8,3)},
\]
\[ 
f_2=x_{(1,3)}x_{(2,1)}x_{(7,4)}x_{(8,2)} - x_{(1,2)}x_{(2,4)}x_{(7,1)}x_{(8,3)},
\]
\[
 f_3=x_{(1,3)}x_{(3,1)}x_{(6,4)}x_{(8,2)} - x_{(1,2)}x_{(3,4)}x_{(6,1)}x_{(8,3)},
 \]
 \[
 f_4= x_{(13)}x_{(2,1)}x_{(6,4)}x_{(8,2)} - x_{(1,2)}x_{(2,4)}x_{(6,1)}x_{(8,3)}.
\]
The four binomials above correspond to the four zig-zag walks drawn in Figure \ref{Fig: 4 zigzag}.
 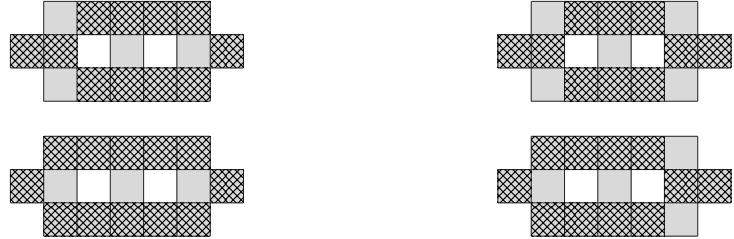
\begin{figure}[h!]
 \centering
 \begin{subfigure}[h]{0.5\textwidth}
\centering
     \resizebox{0.5\textwidth}{!}{
  \begin{tikzpicture}

\draw[thick] (1,0) --  (6,0);
\draw[thick] (0,1) --  (7,1);
\draw[thick] (0,2) --  (7,2);
\draw[thick] (1,3) --  (6,3);

\draw[thick] (0,1) --  (0,2);
\draw[thick] (1,0) --  (1,3);
\draw[thick] (2,0) --  (2,3);
\draw[thick] (3,0) --  (3,3);
\draw[thick] (4,0) --  (4,3);
\draw[thick] (5,0) --  (5,3);
\draw[thick] (6,0) --  (6,3);
\draw[thick] (7,1) --  (7,2);

\fill[fill=gray, fill opacity=0.3] (1,0) -- (6,0)-- (6,1) -- (1,1);
\fill[fill=gray, fill opacity=0.3] (1,2) -- (6,2)-- (6,3) -- (1,3);
\fill[fill=gray, fill opacity=0.3] (0,1) -- (2,1)-- (2,2) -- (0,2);
\fill[fill=gray, fill opacity=0.3] (3,1) -- (4,1)-- (4,2) -- (3,2);
\fill[fill=gray, fill opacity=0.3] (5,1) -- (7,1)-- (7,2) -- (5,2);

\fill[pattern=crosshatch ](2,0) -- (6,0)-- (6,1) -- (2,1);
\fill[pattern=crosshatch ] (0,1) -- (2,1)-- (2,2) -- (0,2);
\fill[pattern=crosshatch ] (6,1) -- (7,1)-- (7,2) -- (6,2);
\fill[pattern=crosshatch ] (2,2) -- (6,2)-- (6,3) -- (2,3);

   \end{tikzpicture}}
 \hfill\\ 
 \hfill\\
    \resizebox{0.5\textwidth}{!}{
  \begin{tikzpicture}

\draw[thick] (1,0) --  (6,0);
\draw[thick] (0,1) --  (7,1);
\draw[thick] (0,2) --  (7,2);
\draw[thick] (1,3) --  (6,3);

\draw[thick] (0,1) --  (0,2);
\draw[thick] (1,0) --  (1,3);
\draw[thick] (2,0) --  (2,3);
\draw[thick] (3,0) --  (3,3);
\draw[thick] (4,0) --  (4,3);
\draw[thick] (5,0) --  (5,3);
\draw[thick] (6,0) --  (6,3);
\draw[thick] (7,1) --  (7,2);

\fill[fill=gray, fill opacity=0.3] (1,0) -- (6,0)-- (6,1) -- (1,1);
\fill[fill=gray, fill opacity=0.3] (1,2) -- (6,2)-- (6,3) -- (1,3);
\fill[fill=gray, fill opacity=0.3] (0,1) -- (2,1)-- (2,2) -- (0,2);
\fill[fill=gray, fill opacity=0.3] (3,1) -- (4,1)-- (4,2) -- (3,2);
\fill[fill=gray, fill opacity=0.3] (5,1) -- (7,1)-- (7,2) -- (5,2);

\fill[pattern=crosshatch ](1,0) -- (6,0)-- (6,1) -- (1,1);
\fill[pattern=crosshatch ](0,1) -- (1,1)-- (1,2) -- (0,2);
\fill[pattern=crosshatch ] (6,1) -- (7,1)-- (7,2) -- (6,2);
\fill[pattern=crosshatch ](1,2) -- (6,2)-- (6,3) -- (1,3);
   \end{tikzpicture}}
\end{subfigure}%
~
 \begin{subfigure}[h]{0.5\textwidth}
 \centering
    \resizebox{0.5\textwidth}{!}{
  \begin{tikzpicture}

\draw[thick] (1,0) --  (6,0);
\draw[thick] (0,1) --  (7,1);
\draw[thick] (0,2) --  (7,2);
\draw[thick] (1,3) --  (6,3);

\draw[thick] (0,1) --  (0,2);
\draw[thick] (1,0) --  (1,3);
\draw[thick] (2,0) --  (2,3);
\draw[thick] (3,0) --  (3,3);
\draw[thick] (4,0) --  (4,3);
\draw[thick] (5,0) --  (5,3);
\draw[thick] (6,0) --  (6,3);
\draw[thick] (7,1) --  (7,2);

\fill[fill=gray, fill opacity=0.3] (1,0) -- (6,0)-- (6,1) -- (1,1);
\fill[fill=gray, fill opacity=0.3] (1,2) -- (6,2)-- (6,3) -- (1,3);
\fill[fill=gray, fill opacity=0.3] (0,1) -- (2,1)-- (2,2) -- (0,2);
\fill[fill=gray, fill opacity=0.3] (3,1) -- (4,1)-- (4,2) -- (3,2);
\fill[fill=gray, fill opacity=0.3] (5,1) -- (7,1)-- (7,2) -- (5,2);

\fill[pattern=crosshatch ] (2,0) -- (5,0)-- (5,1) -- (2,1);
\fill[pattern=crosshatch ] (0,1) -- (2,1)-- (2,2) -- (0,2);
\fill[pattern=crosshatch ] (5,1) -- (7,1)-- (7,2) -- (5,2);
\fill[pattern=crosshatch ](2,2) -- (5,2)-- (5,3) -- (2,3);
   \end{tikzpicture}}
   \hfill\\ 
   \hfill\\
    \resizebox{0.5\textwidth}{!}{
  \begin{tikzpicture}

\draw[thick] (1,0) --  (6,0);
\draw[thick] (0,1) --  (7,1);
\draw[thick] (0,2) --  (7,2);
\draw[thick] (1,3) --  (6,3);

\draw[thick] (0,1) --  (0,2);
\draw[thick] (1,0) --  (1,3);
\draw[thick] (2,0) --  (2,3);
\draw[thick] (3,0) --  (3,3);
\draw[thick] (4,0) --  (4,3);
\draw[thick] (5,0) --  (5,3);
\draw[thick] (6,0) --  (6,3);
\draw[thick] (7,1) --  (7,2);

\fill[fill=gray, fill opacity=0.3] (1,0) -- (6,0)-- (6,1) -- (1,1);
\fill[fill=gray, fill opacity=0.3] (1,2) -- (6,2)-- (6,3) -- (1,3);
\fill[fill=gray, fill opacity=0.3] (0,1) -- (2,1)-- (2,2) -- (0,2);
\fill[fill=gray, fill opacity=0.3] (3,1) -- (4,1)-- (4,2) -- (3,2);
\fill[fill=gray, fill opacity=0.3] (5,1) -- (7,1)-- (7,2) -- (5,2);

\fill[pattern=crosshatch ] (1,0) -- (5,0)-- (5,1) -- (1,1);
\fill[pattern=crosshatch ] (0,1) -- (1,1)-- (1,2) -- (0,2);
\fill[pattern=crosshatch ](5,1) -- (7,1)-- (7,2) -- (5,2);
\fill[pattern=crosshatch ] (1,2) -- (5,2)-- (5,3) -- (1,3);
   \end{tikzpicture}}
\end{subfigure}
\caption{The zig-zag walks related to $f_1,\ldots, f_4$.}\label{Fig: 4 zigzag}
   \end{figure}   
    \begin{figure}[h!]
  \centering
    \resizebox{0.4\textwidth}{!}{
  \begin{tikzpicture}

\draw[thick] (1,0) --  (6,0);
\draw[thick] (0,1) --  (7,1);
\draw[thick] (0,2) --  (7,2);
\draw[thick] (1,3) --  (6,3);
\draw[thick] (2,-1) --  (3,-1);
\draw[thick] (4,-1) --  (5,-1);
\draw[thick] (2,4) --  (3,4);
\draw[thick] (4,4) --  (5,4);

\draw[thick] (0,1) --  (0,2);
\draw[thick] (1,0) --  (1,3);
\draw[thick] (2,-1) --  (2,4);
\draw[thick] (3,-1) --  (3,4);
\draw[thick] (4,-1) --  (4,4);
\draw[thick] (5,-1) --  (5,4);
\draw[thick] (6,0) --  (6,3);
\draw[thick] (7,1) --  (7,2);

\fill[fill=gray, fill opacity=0.3] (1,0) -- (6,0)-- (6,1) -- (1,1);
\fill[fill=gray, fill opacity=0.3] (1,2) -- (6,2)-- (6,3) -- (1,3);
\fill[fill=gray, fill opacity=0.3] (0,1) -- (2,1)-- (2,2) -- (0,2);
\fill[fill=gray, fill opacity=0.3] (3,1) -- (4,1)-- (4,2) -- (3,2);
\fill[fill=gray, fill opacity=0.3] (5,1) -- (7,1)-- (7,2) -- (5,2);

\fill[fill=gray, fill opacity=0.3] (2,-1) -- (3,-1)-- (3,0) -- (2,0);
\fill[fill=gray, fill opacity=0.3] (4,-1) -- (5,-1)-- (5,0) -- (4,0);
\fill[fill=gray, fill opacity=0.3] (2,3) -- (3,3)-- (3,4) -- (2,4);
\fill[fill=gray, fill opacity=0.3] (4,3) -- (5,3)-- (5,4) -- (4,4);
   \end{tikzpicture}}
      \caption{}\label{Fig: 2PF}

   \end{figure}  
In this case, the generators of $J_\MP$ in $J_\MP\setminus I_\MP$ are all related to zig-zag walks. However, we computed $J_\MP$ for the polyomino in Figure \ref{Fig: 2PF}, and we found that there are generators of degree $6$ that are not related to zig-zag walks, for example
\[
g= x_{(1,4)}x_{(3,1)}x_{(4,6)}x_{(5,1)}x_{(6,6)}x_{(8,3)} - x_{(1,3)}x_{(3,6)}x_{(4,1)}x_{(5,6)}x_{(6,1)}x_{(8,4)}.
\]
In Figure \ref{Fig: NoZig} (A), we highlight the intervals related to $g$. On the other hand, there are two zig-zag walks that arises from $g$, as in Figure \ref{Fig: NoZig} (B).

 \begin{figure}[h!]
  \centering
  \begin{subfigure}[h]{0.5 \textwidth}
  \centering
    \resizebox{0.7\textwidth}{!}{
  \begin{tikzpicture}

\draw[thick] (1,0) --  (6,0);
\draw[thick] (0,1) --  (7,1);
\draw[thick] (0,2) --  (7,2);
\draw[thick] (1,3) --  (6,3);
\draw[thick] (2,-1) --  (3,-1);
\draw[thick] (4,-1) --  (5,-1);
\draw[thick] (2,4) --  (3,4);
\draw[thick] (4,4) --  (5,4);

\draw[thick] (0,1) --  (0,2);
\draw[thick] (1,0) --  (1,3);
\draw[thick] (2,-1) --  (2,4);
\draw[thick] (3,-1) --  (3,4);
\draw[thick] (4,-1) --  (4,4);
\draw[thick] (5,-1) --  (5,4);
\draw[thick] (6,0) --  (6,3);
\draw[thick] (7,1) --  (7,2);

\fill[fill=gray, fill opacity=0.3] (1,0) -- (6,0)-- (6,1) -- (1,1);
\fill[fill=gray, fill opacity=0.3] (1,2) -- (6,2)-- (6,3) -- (1,3);
\fill[fill=gray, fill opacity=0.3] (0,1) -- (2,1)-- (2,2) -- (0,2);
\fill[fill=gray, fill opacity=0.3] (3,1) -- (4,1)-- (4,2) -- (3,2);
\fill[fill=gray, fill opacity=0.3] (5,1) -- (7,1)-- (7,2) -- (5,2);

\fill[fill=gray, fill opacity=0.3] (2,-1) -- (3,-1)-- (3,0) -- (2,0);
\fill[fill=gray, fill opacity=0.3] (4,-1) -- (5,-1)-- (5,0) -- (4,0);
\fill[fill=gray, fill opacity=0.3] (2,3) -- (3,3)-- (3,4) -- (2,4);
\fill[fill=gray, fill opacity=0.3] (4,3) -- (5,3)-- (5,4) -- (4,4);

\fill[pattern=crosshatch ] (0,1) -- (2,1)-- (2,2) -- (0,2);
\fill[pattern=crosshatch ] (2,2) -- (3,2)-- (3,4) -- (2,4);
\fill[pattern=crosshatch ] (2,-1) -- (3,-1)-- (3,1) -- (2,1);

\fill[pattern=crosshatch ] (5,1) -- (7,1)-- (7,2) -- (5,2);
\fill[pattern=crosshatch ] (4,2) -- (5,2)-- (5,4) -- (4,4);
\fill[pattern=crosshatch ](4,-1) -- (5,-1)-- (5,1) -- (4,1);

   \end{tikzpicture}}
     \caption{$g$ is not related to a zig-zag walk...}
   \end{subfigure}%
 \begin{subfigure}[h]{0.5 \textwidth}\
\centering
    \resizebox{0.7\textwidth}{!}{
  \begin{tikzpicture}

\draw[thick] (1,0) --  (6,0);
\draw[thick] (0,1) --  (7,1);
\draw[thick] (0,2) --  (7,2);
\draw[thick] (1,3) --  (6,3);
\draw[thick] (2,-1) --  (3,-1);
\draw[thick] (4,-1) --  (5,-1);
\draw[thick] (2,4) --  (3,4);
\draw[thick] (4,4) --  (5,4);

\draw[thick] (0,1) --  (0,2);
\draw[thick] (1,0) --  (1,3);
\draw[thick] (2,-1) --  (2,4);
\draw[thick] (3,-1) --  (3,4);
\draw[thick] (4,-1) --  (4,4);
\draw[thick] (5,-1) --  (5,4);
\draw[thick] (6,0) --  (6,3);
\draw[thick] (7,1) --  (7,2);

\fill[fill=gray, fill opacity=0.3] (1,0) -- (6,0)-- (6,1) -- (1,1);
\fill[fill=gray, fill opacity=0.3] (1,2) -- (6,2)-- (6,3) -- (1,3);
\fill[fill=gray, fill opacity=0.3] (0,1) -- (2,1)-- (2,2) -- (0,2);
\fill[fill=gray, fill opacity=0.3] (3,1) -- (4,1)-- (4,2) -- (3,2);
\fill[fill=gray, fill opacity=0.3] (5,1) -- (7,1)-- (7,2) -- (5,2);

\fill[fill=gray, fill opacity=0.3] (2,-1) -- (3,-1)-- (3,0) -- (2,0);
\fill[fill=gray, fill opacity=0.3] (4,-1) -- (5,-1)-- (5,0) -- (4,0);
\fill[fill=gray, fill opacity=0.3] (2,3) -- (3,3)-- (3,4) -- (2,4);
\fill[fill=gray, fill opacity=0.3] (4,3) -- (5,3)-- (5,4) -- (4,4);

\fill[pattern=north east lines] (0,1) -- (2,1)-- (2,2) -- (0,2);
\fill[pattern=north east lines] (2,2) -- (3,2)-- (3,4) -- (2,4);
\fill[pattern=north east lines](2,-1) -- (3,-1)-- (3,1) -- (2,1);
\fill[pattern=north east lines] (3,1) -- (4,1)-- (4,2) -- (3,2);

\fill[pattern=north west lines] (5,1) -- (7,1)-- (7,2) -- (5,2);
\fill[pattern=north west lines](4,2) -- (5,2)-- (5,4) -- (4,4);
\fill[pattern=north west lines](4,-1) -- (5,-1)-- (5,1) -- (4,1);
\fill[pattern=north west lines] (3,1) -- (4,1)-- (4,2) -- (3,2);
   \end{tikzpicture}}
   \caption{...but there are two zig-zag walks}
   \end{subfigure}   
   \caption{}\label{Fig: NoZig}
   \end{figure}

\end{example}

%% file: Algorithm.tex
To verify that the non-existence of zig-zag walk is a sufficient condition for the primality of $I_\MP$, for any multiply connected polyomino $\MP$ of rank $\leq 14$, is not an easy task.  In fact, the set of polyominoes grows exponentially with respect to the rank as the following table, obtained by the implementation  in \cite{MRR}, shows. 
\begin{table}[h]
\centering
\begin{tabular}{lrrrrrrrr}
Rank  & 7 & 8 & 9 & 10 & 11 & 12 & 13 & 14 \\
Multiply connected polyominoes & 1 & 6 & 37 & 195 & 979 & 4663 & 21474 & 96496 
\end{tabular}
\end{table}

\begin{theorem}\label{Theo: rank <= 14}
 Let $\MP$ be a polyomino with $\rank(\MP)$ $\leq 14$. The following conditions are equivalent:
 \begin{enumerate}    
\item the polyomino ideal $I_\MP$ is prime;
\item $\MP$ contains no zig-zag walks.
 \end{enumerate}
\end{theorem}
\begin{proof}
$(1) \Rightarrow (2)$ It is an immediate consequence of Corollary \ref{Cor: if zig-zag then non-prime}. \\
$(2) \Rightarrow (1)$ To prove the claim we have implemented a computer program that performs the following $3$ steps:
 \begin{enumerate}
 \item[(S1)] Compute the set of all multiply connected polyominoes with rank $\leq 14$, namely $P$.
 \item [(S2)] Compute the set of polyominoes  $\mathrm{NP}\subset P$ whose associated ideals are not primes. We used a routine developed in \texttt{Macaulay2} (see \cite{M2}).
 \item [(S3)] Verify that all polyominoes in $\mathrm{NP}$ have at least one zig-zag walk.
 \end{enumerate}
 We refer to \cite{MRR} for a complete description of the algorithm that we used.

\end{proof}  

%% file: NewProofFG.tex
\section{Grid Polyominoes}\label{sec:grid}
From a view point of finding a new class of prime polyomino ideals, due to Corollary \ref{Cor: if zig-zag then non-prime}, it is reasonable to consider multiply connected polyominoes with no zig-zag walks.  In this section, we consider polyominoes obtained subtracting some inner intervals by a given interval of $\NN^2$, similarly as done in \cite{HQ} and  \cite{Sh1}. But, if the cells are removed without a specific pattern, one can easily obtain a zig-zag walk in this case, too (see Figure \ref{fig:grid}(B)). Hence, we define an infinite family of polyominoes with no zig-zag walks by their intrinsic shape: the grid polyominoes.

\begin{definition}\label{def: grid}
Let $\MP \subseteq I:=[(1,1),(m,n)]$ be a polyomino such that
\[
\MP=I \setminus \{\MH_{ij}: i \in [r], \ j \in [s] \},
\]
where $\MH_{ij}=[a_{ij},b_{ij}]$, with $a_{ij} =((a_{ij})_1,( a_{ij})_2)$, $b_{ij} =((b_{ij})_1,( b_{ij})_2)$, $1<(a_{ij})_1 < (b_{ij})_1 < m$,  $1<(a_{ij})_2 < (b_{ij})_2 < n$, and 
\begin{enumerate}
\item for any $i \in [r]$ and $\ell,k \in [s]$ we have $(a_{i\ell})_1=(a_{ik})_1$ and $(b_{i\ell})_1=(b_{ik})_1$;
\item for any $j \in [s]$ and $\ell,k \in [r]$ we have $(a_{\ell j})_2=(a_{kj})_2$ and $(b_{\ell j})_2=(b_{kj})_2$;
\item for any $i \in [r-1]$ and $\ j\in [s-1]$, we have $(a_{i+1j})_1 = (b_{ij})_1 +1 $ and $(a_{ij+1})_2 = (b_{ij})_2 +1$.
\end{enumerate}
We call $\MP$ a \emph{grid polyomino}.
\end{definition}

  \begin{figure}[h!]
  \centering
    \resizebox{0.9\textwidth}{!}{
  \begin{tikzpicture}
  
\draw[thick] ( 0 ,0) --  ( 0 ,8);
\draw[thick] ( 1 ,0) --  ( 1 ,8);
\draw[thick] ( 2 ,0) --  ( 2 ,8);
\draw[thick] ( 3 ,0) --  ( 3 ,8);
\draw[thick] ( 4 ,0) --  ( 4 ,8);
\draw[thick] ( 5 ,0) --  ( 5 ,8);
\draw[thick] ( 6 ,0) --  ( 6 ,8);
\draw[thick] ( 7 ,0) --  ( 7 ,8);
\draw[thick] ( 8 ,0) --  ( 8 ,8);
\draw[thick] ( 9 ,0) --  ( 9 ,8);
\draw[thick] ( 10 ,0) --  ( 10 ,8);

\draw[thick] (0, 0 ) --  (10, 0 );
\draw[thick] (0, 1 ) --  (10, 1 );
\draw[thick] (0, 2 ) --  (10, 2 );
\draw[thick] (0, 3 ) --  (10, 3 );
\draw[thick] (0, 4 ) --  (10, 4 );
\draw[thick] (0, 5 ) --  (10, 5 );
\draw[thick] (0, 6 ) --  (10, 6 );
\draw[thick] (0, 7 ) --  (10, 7 );
\draw[thick] (0, 8 ) --  (10, 8);

\fill[fill=gray, fill opacity=0.3] (0,0) -- (10,0)-- (10,1) -- (0,1);
\fill[fill=gray, fill opacity=0.3] (0,7) -- (10,7)-- (10,8) -- (0,8);
\fill[fill=gray, fill opacity=0.3] (0,3) -- (10,3)-- (10,4) -- (0,4);
\fill[fill=gray, fill opacity=0.3] (0,1) -- (1,1)-- (1,3) -- (0,3);
\fill[fill=gray, fill opacity=0.3] (4,1) -- (5,1)-- (5,3) -- (4,3);
\fill[fill=gray, fill opacity=0.3] (6,1) -- (7,1)-- (7,3) -- (6,3);
\fill[fill=gray, fill opacity=0.3] (9,1) -- (10,1)-- (10,3) -- (9,3);
\fill[fill=gray, fill opacity=0.3] (0,4) -- (1,4)-- (1,7) -- (0,7);
\fill[fill=gray, fill opacity=0.3] (4,4) -- (5,4)-- (5,7) -- (4,7);
\fill[fill=gray, fill opacity=0.3] (6,4) -- (7,4)-- (7,7) -- (6,7);
\fill[fill=gray, fill opacity=0.3] (9,4) -- (10,4)-- (10,7) -- (9,7);


\draw[thick] (0+17 ,0) --  ( 0+17 ,5*1.6);
\draw[thick] ( 1.6+17 ,0) --  ( 1.6+17 ,5*1.6);
\draw[thick] ( 2*1.6+17 ,0) --  ( 2*1.6+17 ,5*1.6);
\draw[thick] ( 3*1.6+17 ,0) --  ( 3*1.6+17 ,5*1.6);
\draw[thick] ( 4*1.6+17 ,0) --  ( 4*1.6+17 ,5*1.6);
\draw[thick] ( 5*1.6+17 ,0) --  ( 5*1.6+17 ,5*1.6);

\draw[thick] (0+17, 0 ) --  (5*1.6+17, 0 );
\draw[thick] (0+17, 1*1.6 ) --  (5*1.6+17, 1 *1.6);
\draw[thick] (0+17, 2 *1.6) --  (5*1.6+17, 2 *1.6);
\draw[thick] (0+17, 3 *1.6) --  (5*1.6+17, 3*1.6 );
\draw[thick] (0+17, 4*1.6 ) --  (5*1.6+17, 4*1.6 );
\draw[thick] (0+17, 5 *1.6) --  (5*1.6+17, 5 *1.6);

\draw (5,-1) circle (0pt) node [anchor=north] {{\Huge(A) A grid polyomino}};

\fill[fill=gray, fill opacity=0.3] (0+17,0) -- (5*1.6+17,0)-- (5*1.6+17,1*1.6) -- (0+17,1*1.6);
\fill[fill=gray, fill opacity=0.3] (0+17,4*1.6) -- (5*1.6+17,4*1.6)-- (5*1.6+17,5*1.6) -- (0+17,5*1.6);
\fill[fill=gray, fill opacity=0.3] (0+17,1*1.6) -- (1*1.6+17,1*1.6)-- (1*1.6+17,2*1.6) -- (0+17,2*1.6);
\fill[fill=gray, fill opacity=0.3] (2*1.6+17,1*1.6) -- (3*1.6+17,1*1.6)-- (3*1.6+17,2*1.6) -- (2*1.6+17,2*1.6);
\fill[fill=gray, fill opacity=0.3] (4*1.6+17,1*1.6) -- (5*1.6+17,1*1.6)-- (5*1.6+17,2*1.6) -- (4*1.6+17,2*1.6);
\fill[fill=gray, fill opacity=0.3] (0+17,2*1.6) -- (2*1.6+17,2*1.6)-- (2*1.6+17,3*1.6) -- (0+17,3*1.6);
\fill[fill=gray, fill opacity=0.3] (3*1.6+17,2*1.6) -- (5*1.6+17,2*1.6)-- (5*1.6+17,3*1.6) -- (3*1.6+17,3*1.6);
\fill[fill=gray, fill opacity=0.3] (0+17,3*1.6) -- (1*1.6+17,3*1.6)-- (1*1.6+17,4*1.6) -- (0+17,4*1.6);
\fill[fill=gray, fill opacity=0.3] (2*1.6+17,3*1.6) -- (3*1.6+17,3*1.6)-- (3*1.6+17,4*1.6) -- (2*1.6+17,4*1.6);
\fill[fill=gray, fill opacity=0.3] (4*1.6+17,3*1.6) -- (5*1.6+17,3*1.6)-- (5*1.6+17,4*1.6) -- (4*1.6+17,4*1.6);
\draw (0+21,-1) circle (0pt) node [anchor=north] {{\Huge (B) A non-grid polyomino, with a zig-zag walk.}};

   \end{tikzpicture}}
   \caption{} \label{fig:grid}
   \end{figure}

Let $\MP$ be a grid polyomino and let $T_{\MP}$ and $J_{\MP}$ be the toric ring and the toric ideal associated to $\MP$, respectively, as defined in Section \ref{sec: toric}, where the hole $H_{ij}$ induces the subset $\MF_{i,j}$ and the variable $\omega_{i,j}$.  We claim that the grid polyominoes are primes. In order to prove this,  we are going to show that $I_{\MP} = J_{\MP}$. 

Let $f = f^+ - f^- \in J_{\MP}$, we define $V_+ = \{ v \in V(\MP) \mid x_v \text{ divides } f^+\}$, and, similarly, $V_- = \{ v \in V(\MP) \mid x_v \text{ divides } f^-\}$. A binomial $f$ in a binomial ideal $J$ is said to be \textit{redundant} if it can be expressed as a linear combination of binomials in $J$ of lower degree. A binomial is said to be \textit{irredundant} if it is not redundant. The following lemma, that has been stated in \cite{Sh1} but only for a family of polyominoes, holds also for any $J_{\MP}$, as defined in Section \ref{sec: toric}. Even if the proof is essentially the same of \cite[Lemma 2.2]{Sh1}, we report it for the sake of completeness.
\begin{lemma}\label{lemma: f red}
Let $f = f^+ - f^- \in J_{\MP}$ be a binomial of degree $\geq 3$. If there exist three vertices $p,q \in V_+$ and $r \in V_-$ such that $p,q$ are diagonal (resp. anti-diagonal) corners of an inner interval of $\MP$ and $r$ is one of the anti-diagonal (resp. diagonal) corners of the inner interval, then $f$ is redundant in $J_{\MP}$.
\end{lemma}
\begin{proof}
Let $s$ be the other corner of the inner interval determined by $p,q$ and $r$. Then 
\begin{align*}
f &= f^+ - f^- = x_px_q \frac{f^+}{x_px_q} - f^- \\
&= (x_px_q - x_rx_s) \frac{f^+}{x_px_q} + x_rx_s \frac{f^+}{x_px_q} - f^- \\
&=  (x_px_q - x_rx_s) \frac{f^+}{x_px_q} + x_r \left(x_s \frac{f^+}{x_px_q} - \frac{f^-}{x_r}\right).
\end{align*}
By Lemma \ref{Lemma: I_P subseteq J_P}, it holds $I_{\MP} \subseteq J_{\MP}$. Since $x_px_q - x_rx_s \in I_{\MP} \subseteq J_{\MP}$, and  $J_{\MP}$ is a prime ideal, then $x_s \frac{f^+}{x_px_q} - \frac{f^-}{x_r} \in J_{\MP}$, and the statement is proved. 
\end{proof}

Let $\MP$ be a grid polyomino, and let $\mathcal{H}_{ij}$, for $i\in [r]$ and $j \in [s]$, be its holes, enumerated as in Definition \ref{def: grid}. Fix $i \in [r]$ and $j \in [s]$, we denote by $\ML_{i,j}$ the set 

\[
\ML_{i,j} = \mathcal{F}_{i,j} \setminus \bigcup_{\substack{k \leq i \\ h \leq j \\ (h,k) \neq (i,j)}} \mathcal{F}_{h,k}.
\] 
   
In Figure \ref{fig: L22}, it is displayed an example of a set $\ML_{i,j}$. In particular, for the grid polyomino $\MP$ in figure, $\ML_{2,2}$ consists of all vertices of $\MP$ in the dark grey region.
   
 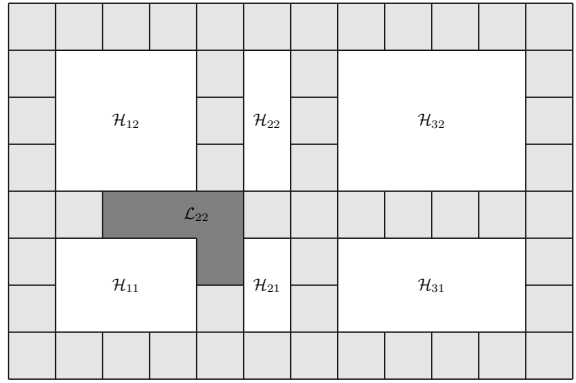
\begin{figure}[h!]
  \centering
\resizebox{0.6\textwidth}{!}{
  \begin{tikzpicture}
  \fill[fill=gray, fill opacity=1] (2,3) -- (2,4)-- (5,4) -- (5,3);
\fill[fill=gray, fill opacity=1] (4,2) -- (5,2)-- (5,3) -- (4,3);
\draw[thick] ( 0 ,0) --  ( 0 ,8);
\draw[thick] ( 1 ,0) --  ( 1 ,8);
\draw[thick] ( 2 ,0) --  ( 2 ,1);
\draw[thick] ( 2 ,3) --  ( 2 ,4);
\draw[thick] ( 2 ,7) --  ( 2 ,8);
\draw[thick] ( 3 ,0) --  ( 3,1);
\draw[thick] ( 3 ,7) --  ( 3 ,8);
\draw[thick] ( 4 ,0) --  ( 4 ,3);
\draw[thick] ( 4 ,4) --  ( 4 ,8);
\draw[thick] ( 5 ,0) --  ( 5 ,8);
\draw[thick] ( 6 ,0) --  ( 6 ,8);
\draw[thick] ( 7 ,0) --  ( 7 ,8);
\draw[thick] ( 8 ,0) --  ( 8 ,1);
\draw[thick] ( 8 ,3) --  ( 8 ,4);
\draw[thick] ( 8 ,7) --  ( 8 ,8);
\draw[thick] ( 9 ,0) --  ( 9 ,1);
\draw[thick] ( 9 ,3) --  ( 9 ,4);
\draw[thick] ( 9 ,7) --  ( 9 ,8);
\draw[thick] ( 10 ,0) --  ( 10 ,1);
\draw[thick] ( 10 ,3) --  ( 10 ,4);
\draw[thick] ( 10 ,7) --  ( 10 ,8);
\draw[thick] ( 11 ,0) --  ( 11 ,8);
\draw[thick] ( 12 ,0) --  ( 12 ,8);

\draw[thick] (0, 0 ) --  (12, 0 );
\draw[thick] (0, 1 ) --  (12, 1 );
\draw[thick] (0, 2 ) --  (1, 2 );
\draw[thick] (4, 2 ) --  (5, 2 );
\draw[thick] (6, 2 ) --  (7, 2 );
\draw[thick] (11, 2 ) --  (12, 2 );
\draw[thick] (0, 3 ) --  (4, 3 );
\draw[thick] (12, 3 ) --  (5, 3 );
\draw[thick] (0, 4 ) --  (12, 4 );
\draw[thick] (0, 5 ) --  (1, 5 );
\draw[thick] (4, 5 ) --  (5, 5 );
\draw[thick] (6, 5 ) --  (7, 5 );
\draw[thick] (11, 5 ) --  (12, 5 );
\draw[thick] (0, 6 ) --  (1, 6 );
\draw[thick] (4, 6 ) --  (5, 6 );
\draw[thick] (6, 6 ) --  (7, 6 );
\draw[thick] (11, 6 ) --  (12, 6 );
\draw[thick] (0, 7 ) --  (12, 7 );
\draw[thick] (0, 8 ) --  (12, 8);

\fill[fill=gray, fill opacity=0.2] (0,0) -- (12,0)-- (12,1) -- (0,1);
\fill[fill=gray, fill opacity=0.2] (0,7) -- (12,7)-- (12,8) -- (0,8);
\fill[fill=gray, fill opacity=0.2] (0,3) -- (12,3)-- (12,4) -- (0,4);
\fill[fill=gray, fill opacity=0.2] (0,1) -- (1,1)-- (1,3) -- (0,3);
\fill[fill=gray, fill opacity=0.2] (4,1) -- (5,1)-- (5,3) -- (4,3);
\fill[fill=gray, fill opacity=0.2] (6,1) -- (7,1)-- (7,3) -- (6,3);
\fill[fill=gray, fill opacity=0.2] (11,1) -- (12,1)-- (12,3) -- (11,3);
\fill[fill=gray, fill opacity=0.2] (0,4) -- (1,4)-- (1,7) -- (0,7);
\fill[fill=gray, fill opacity=0.2] (4,4) -- (5,4)-- (5,7) -- (4,7);
\fill[fill=gray, fill opacity=0.2] (6,4) -- (7,4)-- (7,7) -- (6,7);
\fill[fill=gray, fill opacity=0.2] (11,4) -- (12,4)-- (12,7) -- (11,7);

\node at (2.5,2) {$\mathcal{H}_{11}$};
\node at (5.5,2) {$\mathcal{H}_{21}$};
\node at (9,2) {$\mathcal{H}_{31}$};
\node at (2.5,5.5) {$\mathcal{H}_{12}$};
\node at (5.5,5.5) {$\mathcal{H}_{22}$};
\node at (9,5.5) {$\mathcal{H}_{32}$};
\node at (4,3.5) {$\mathcal{L}_{22}$};
   \end{tikzpicture}}
   \caption{An example of $\mathcal{L}_{i,j}$.}\label{fig: L22}
   \end{figure}      
   
\begin{lemma}\label{lemma: if v in L, then w in L}
Let $\MP$ be a grid polyomino. Let $f = f^+ - f^- \in J_{\MP}$. If $v \in V_+ \cap \ML_{i,j}$, for some $i\in [r]$ and $j \in [s]$,  then there exists $v' \in V_- \cap \ML_{i,j}$.
\end{lemma}

\begin{proof}
We prove the assertion showing that for all $(i,j)$ and any $v\in \ML_{i,j}$ with $v\in V_+$, there exists $v'\in V_-$ such that $v'\in \ML_{i,j}$. 
Let 
\[
(i_1, j_1) = \min \{(k,h) \mid  V_+ \cap \MF_{k,h}\neq \emptyset\}.
\]
If such a pair does not exist, there is nothing to prove. Otherwise, let $v_1 \in  V_+ \cap \ML_{i_1, j_1}$. Since $\omega_{i_1,j_1} \mid \varphi(f^+)$, then $\omega_{i_1,j_1} \mid \varphi(f^-)$. It follows there exists $v_1' \in V_- \cap \MF_{i_1, j_1}$. By the minimality of the pair $(i_1, j_1)$ and since $\varphi(f^+) = \varphi(f^-)$, $v_1' \in \ML_{i_1,j_1}$. Let 
\[
(i_2, j_2) = \min \{(k,h) \mid  (V_+\setminus \{v_1\}) \cap \MF_{k,h}\neq \emptyset\}.
\]
If such a pair does not exist, we have done. Otherwise, let $v_2 \in  (V_+\setminus \{v_1\}) \cap \ML_{i_2, j_2}$. 
We observe that because of the existence of $v_1$ and $v_1'$ we have the following equation
\[
 f=(\prod_{\substack{k \geq i_1\\ h \geq j_1}} \omega_{k,h}) g,
\]
where we have collected all $\omega_{k,h}$'s induced by $v_1$ and $v_1'$.
Because of the existence of $v_2$, we have that
\[
\omega_{i_2,j_2} \mid \varphi(g^+)= \varphi(g^-).
\] 
It follows there exists $v_2' \in (V_- \setminus \{v_1'\}) \cap \MF_{i_2, j_2}$. By the minimality of the pair $(i_2, j_2)$,  $v_2' \in \ML_{i_2,j_2}$. Iterating this procedure, the assertion follows.

\end{proof}

\begin{theorem}
Let $\MP$ be a grid polyomino. Then $I_{\MP} = J_{\MP}$. 
\end{theorem}

\begin{proof}
By Lemma \ref{Lemma: I_P subseteq J_P}, $I_{\MP} \subseteq J_{\MP}$. We have to prove the opposite inclusion, that is $J_{\MP} \subseteq I_{\MP}$. Since $(J_{\MP})_2 = I_{\MP}$, it suffices to prove that any irredundant binomial of $J_{\MP}$ is of degree 2. Let $f = f^+ - f^- \in J_{\MP}$, with $\deg(f) \geq 3$. Assume by contradiction that $f$ is irredundant. First, we show that there is no $v \in (V_+ \cup V_-) \cap \MF$, where $\MF = \bigcup_{i \in [r], j \in [s]} \MF_{i,j}$.  Assume by contradiction that there exists $v_1 \in (V_+ \cup V_-) \cap \MF$. In particular, $v_1 \in \ML_{i_1, j_1}$, for some $i_1 \in [r], j_1 \in [s]$. Without loss of generality, we may assume $v_1 \in V_+$. By Lemma \ref{lemma: if v in L, then w in L}, there exists $v_1' \in V_- \cap \ML_{i_1, j_1}$. Note that, by the condition (3) in Definition \ref{def: grid}, $v_1$ belongs to $V(\MP) \cap V(\MH_{ij})$, for some $i,j$. The same holds for $v_1'$. Assume $v_1 < v_1'$.  We have the following $3$ cases:
\begin{enumerate}
 \item[(1)] $v_1$ and $v_1'$ belong to the same maximal vertical (resp. horizontal) edge interval;
 \item[(2A)] at least one between $v_1$ and $v_1'$ is not a corner of an hole of $\MP$ (e.g., see Figure \ref{fig:v_1 and v_1'} (A));
  \item [(2B)] $v_1$ and $v_1'$ are both diagonal (or anti-diagonal) corners of some holes of $\MP$ (e.g., see Figure \ref{fig:v_1 and v_1'} (B)).
\end{enumerate}

 \ \\

(1) If $v_1$ and $v_1'$ belong to the same maximal vertical edge interval, there exists $v_2' \in V_-$ that lies on the same maximal horizontal edge interval of $v_1$. The vertices $v_1, v_1'$ and $v_2'$ are corners of an inner interval of $\MP$, and by Lemma \ref{lemma: f red}, $f$ is redundant, which is a contradiction. Similarly, one shows that $v_1$ and $v_1'$ do not belong to the same maximal horizontal edge interval. \\


\begin{figure}[h!]
\begin{subfigure}{0.5\textwidth}
\resizebox{0.9\textwidth}{!}{

\begin{tikzpicture}
\draw[dashed] (0,3)--(1,3);
\draw (1,3)--(3,3);
\draw[dashed] (3,0)--(3,1);
\draw (3,1)--(3,3);
\draw[dashed] (0,4)--(1,4);
\draw (1,4)--(3,4);
\draw[dashed] (3,6)--(3,7);
\draw (3,4)--(3,6);

\draw[dashed] (6,3)--(7,3);
\draw (4,3)--(6,3);
\draw[dashed] (4,0)--(4,1);
\draw (4,1)--(4,3);
\draw[dashed] (6,4)--(7,4);
\draw (4,4)--(6,4);
\draw[dashed] (4,7)--(4,6);
\draw (4,4)--(4,6);

\filldraw (2,4) circle (2.5pt) node[anchor=south]{$v_1$};
\filldraw (4,2) circle (2.5pt) node[anchor=west]{$v'_1$};
\node at (1.5,1.5) {$\mathcal{H}_{i_{1}-1 j_{1}-1}$};
\node at (1.5,5.5) {$\mathcal{H}_{i_{1}-1 j_{1}}$};
\node at (5.5,5.5) {$\mathcal{H}_{i_{1} j_{1}}$};
\node at (5.5,1.5) {$\mathcal{H}_{i_{1} j_{1}-1}$};
\end{tikzpicture}}\caption{}
\end{subfigure}%
\begin{subfigure}{0.5\textwidth}
\resizebox{0.9\textwidth}{!}{
\begin{tikzpicture}
\draw[dashed] (0,3)--(1,3);
\draw (1,3)--(3,3);
\draw[dashed] (3,0)--(3,1);
\draw (3,1)--(3,3);
\draw[dashed] (0,4)--(1,4);
\draw (1,4)--(3,4);
\draw[dashed] (3,6)--(3,7);
\draw (3,4)--(3,6);

\draw[dashed] (6,3)--(7,3);
\draw (4,3)--(6,3);
\draw[dashed] (4,0)--(4,1);
\draw (4,1)--(4,3);
\draw[dashed] (6,4)--(7,4);
\draw (4,4)--(6,4);
\draw[dashed] (4,7)--(4,6);
\draw (4,4)--(4,6);

\filldraw (3,4) circle (2.5pt) node[anchor=west]{$v_1$};
\filldraw (4,3) circle (2.5pt) node[anchor=south]{$v'_1$};
\node at (1.5,1.5) {$\mathcal{H}_{i_{1}-1 j_{1}-1}$};
\node at (1.5,5.5) {$\mathcal{H}_{i_{1}-1 j_{1}}$};
\node at (5.5,5.5) {$\mathcal{H}_{i_{1} j_{1}}$};
\node at (5.5,1.5) {$\mathcal{H}_{i_{1} j_{1}-1}$};
\end{tikzpicture}}\caption{}
\end{subfigure}\caption{}\label{fig:v_1 and v_1'}
\end{figure}

(2A) We assume that at least one between $v_1$ and $v_1'$ is not a corner of an hole of $\MP$, we say $v_1$. Denote by $v_2'$ and $v_3'$ the vertices in $V_-$ that belong to the same horizontal and vertical edge interval of $v_1$, respectively. The vertices $v_1, v_2', v_3'$ are corners of an inner interval of $\MP$, hence by applying Lemma \ref{lemma: f red}  to  $v_1, v_2', v_3'$ we obtain that $f$ is redundant, which is a contradiction. 

(2B) We denote by $v_2'$ the vertex in $V_-$ that belongs to the same vertical edge interval of $v_1$. The vertices $v_1'$ and $v_2'$ are diagonal (or anti-diagonal) corners of an inner interval of $\MP$. Denote by $g,h$ the other two corners, where $g$ is the one on the same horizontal edge interval of $v_1'$. Then the binomial $x_{v_1'}x_{v_2'} - x_gx_h \in J_{\MP}$, and 
\begin{align*}
f &= f^+-f^- = f^+ - x_{v_1'}x_{v_2'} \frac{f^-}{x_{v_1'}x_{v_2'}}  \\
&= f^+ - x_hx_g \left(\frac{f^-}{x_{v_1'}x_{v_2'}}\right)  - (x_{v_1'}x_{v_2'} - x_gx_h)\frac{f^-}{x_{v_1'}x_{v_2'}} \\
&= f' - (x_{v_1'}x_{v_2'} - x_gx_h)\frac{f^-}{x_{v_1'}x_{v_2'}} . 
\end{align*}
Let $v_3'$ be the vertex in $V_-$ that belongs to the same horizontal edge interval of $v_1$. The vertices $v_1, v_3'$, and $g$ are corners of an inner interval of $\MP$. Since $f' \in J_{\MP}$, by applying Lemma \ref{lemma: f red} to $v_1, v_3'$ and $g$, we obtain that $f'$ is redundant, and then also $f$ is redundant, which is a contradiction. 

It follows that the vertices appearing in $V_+ \cup V_-$ do not belong to $\MF$. This means $f \in J_{\MP} \cap \mathbb{K}[x_v \mid v \in V(\MP) \setminus \MF]$. Let $\MP'$ be the subpolyomino of $\MP$ which consists of all cells of $\MP$ having no vertices belonging to $\MF$. $\MP'$ is a simple polyomino and $I_{\MP'} = I_{\MP} \cap \mathbb{K}[x_v \mid v \in V(\MP) \setminus \MF]$. Note that $\alpha(v)$, for every $v \in V(\MP)\setminus \MF$, is a monomial of degree 2 determined by the maximal horizontal and vertical edge intervals to which $v$ belongs. Then, by \cite[Theorem 2.2]{QSS}, $I_{\MP'} = J_{\MP'}= J_{\MP} \cap \mathbb{K}[x_v \mid v \in V(\MP) \setminus \MF]$. Hence, if $f$ is irredundant in $J_{\MP}$, then it is also irredundant in $J_{\MP} \cap \mathbb{K}[x_v \mid v \in V(\MP) \setminus \MF]$. But $I_{\MP'}$ is generated by binomials of degree 2, then $f$ is redundant in $I_{\MP'}$, and then in $J_{\MP} \cap \mathbb{K}[x_v \mid v \in V(\MP) \setminus \MF]$, which is a contradiction. 
\end{proof}

\begin{corollary}\label{cor: grid prime}
Let $\MP$ be a grid polyomino. Then $I_{\MP}$ is prime. 
\end{corollary}

From the main results of this work, that are Corollary \ref{Cor: if zig-zag then non-prime}, Theorem \ref{Theo: rank <= 14} and Corollary \ref{cor: grid prime}, it arises naturally the following:

\begin{conjecture}
  Let $\MP$ be a polyomino. The following conditions are equivalent:
  \begin{enumerate}      
    \item the polyomino ideal $I_\MP$ is prime;
  \item $\MP$ contains no zig-zag walks.
\end{enumerate}
\end{conjecture}